\documentclass[11pt,leqno]{article}
\usepackage{graphicx, amsfonts, amsthm, amsxtra, amssymb, verbatim, makeidx}
\usepackage{subeqnarray, relsize}
\usepackage[mathscr]{euscript}
\usepackage[english]{babel}
\usepackage[fixlanguage]{babelbib}

\usepackage[utf8]{inputenc}
\usepackage[english]{babel}

\usepackage{wrapfig}
\usepackage{amssymb, amsmath, amsthm}
\usepackage{graphicx}
\usepackage{color}
\usepackage{amssymb}
\usepackage{url}
\usepackage{pdfpages}
\usepackage{fancyhdr}
\usepackage{subfig}
\usepackage{titlesec}
\usepackage{enumerate}
\usepackage{comment}
\usepackage{bigints}
\usepackage{diagbox}
\usepackage{cite}
\usepackage[fixlanguage]{babelbib}
\usepackage[unicode,
            psdextra,
            colorlinks=true,
            linkcolor=blue,
            citecolor=green
            ]{hyperref}
\usepackage[nameinlink,capitalize,noabbrev]{cleveref}

\makeindex
\makeglossary

\theoremstyle{definition}
\newtheorem{definition}{\bf Definition}[section]
\newtheorem{theorem}{Theorem}[section]
\newtheorem{lemma}{Lemma}[section]

\newtheorem{proposition}{Proposition}[section]

\newtheorem{remark}{Remark}[section]

\renewenvironment{proof}{{\bfseries \noindent Proof} }{ \qed \\}

\begin{document}

\def\R{\mathbb{R}}                   
\def\Z{\mathbb{Z}}                   
\def\Q{\mathbb{Q}}                   
\def\C{\mathbb{C}}                   
\def\N{\mathbb{N}}                   
\def\uhp{{\mathbb H}}                
\def\A{\mathbb{A}}

\def\P{\mathbb{P}}
\def\Gal{\text{Gal}}
\def\GHod{\text{GHod}}
\def\SHod{\text{SHod}}
\def\Hod{\text{Hod}}
\def\res{\text{res}}
\def\prim{\text{prim}}
\def\dR{\text{dR}}

\def\ker{{\rm ker}}              
\def\GL{{\rm GL}}                
\def\ker{{\rm ker}}              
\def\coker{{\rm coker}}          
\def\im{{\rm Im}}               
\def\coim{{\rm Coim}}            

\def\End{{\rm End}}              
\def\rank{{\rm rank}}                
\def\gcd{{\rm gcd}}                  

\begin{center}
{\LARGE\bf On fake linear cycles inside Fermat varieties
}
\footnote{ 
Math. classification: 14C25, 14C30, 14D07
}
\\
\vspace{.25in} {\large {{\sc Jorge Duque Franco}}\footnote{
Universidad de Chile, Departamento de Matmáticas, Campus Juan Gómez Millas, Las Palmeras 3425, Santiago, Chile,
{\tt georgy11235@gmail.com}}{\sc and Roberto Villaflor Loyola }}\footnote{Pontificia Universidad Católica de Chile, Facultad de Matemáticas, Campus San Joaquín, Avenida Vicuña Mackenna 4860, Santiago, Chile,
{\tt roberto.villaflor@mat.uc.cl}}
\end{center}


\begin{abstract}
We introduce a new class of Hodge cycles with non-reduced associated Hodge loci, we call them fake linear cycles. We characterize them for all Fermat varieties and show that they exist only for degrees $d=3,4,6$, where there are infinitely many in the space of Hodge cycles. These cycles are pathological in the sense that the Zariski tangent space of their associated Hodge locus is of maximal dimension, contrary to a conjecture of Movasati. Moreover they provide examples of algebraic cycles not generated by their periods in the sense of Movasati-Sert\"oz. To study them we compute their Galois action in cohomology and their second order invariant of the IVHS. We conclude that for any degree $d\ge 2+\frac{6}{n}$, the minimal codimension component of the Hodge locus passing through the Fermat variety is the one parametrizing hypersurfaces containing linear subvarieties of dimension $\frac{n}{2}$, extending results of Green, Voisin, Otwinowska and the second author.
\end{abstract}

\section{Introduction}
\label{intro}
The classical Noether-Lefschetz locus $\text{NL}_d$ is the space of degree $d\ge 4$ surfaces in $\P^3$ with Picard rank bigger than 1. This space is known to have countably many components given by algebraic subvarieties of the space of smooth degree $d$ surfaces in $\P^3$. A classical result due to Green \cite{green1988} and Voisin \cite{voisin1988} states that for $d\ge 5$ it has only one minimal codimension component, which parametrizes surfaces containing lines (for $d=4$ all components have the same codimension). The higher dimension analogue of the Noether-Lefschetz locus is the so called Hodge locus $\text{HL}_{n,d}$ which is the locus of degree $d$ hypersurfaces $X\subseteq \mathbb{P}^{n+1}$ for $n$ even, with lattice of Hodge cycles $H^{\frac{n}{2},\frac{n}{2}}(X)\cap H^n(X,\Z)$ of rank bigger than 1. This space is non-trivial for $d\ge 2+\frac{4}{n}$, and it is known to have countably many components which are algebraic subvarieties of $T\subseteq H^0(\P^{n+1},\mathcal{O}(d))$ the space of smooth degree $d$ hypersurfaces of $\P^{n+1}$. A natural question is to ask whether the analogue of Green-Voisin theorem still holds for higher dimensions, i.e. if for $d\ge 2+\frac{6}{n}$ the only minimal codimension component of the Hodge locus is $\Sigma_{(1,\ldots,1)}$ the one parametrizing hypersurfaces containing linear subvarieties of dimension $\frac{n}{2}$. The first result in this direction was obtained by Otwinowska \cite[Theorem 3]{Otw02} who answered positively the question for $d\gg n$. The conjecture for smaller degrees remains open, and even to establish that the codimension of $\Sigma_{(1,\ldots,1)}$ (which is equal to ${\frac{n}{2}+d\choose d}-(\frac{n}{2}+1)^2$) is a lower bound for the codimension of all components is also a conjecture. A partial result on the lower bound conjecture was obtained by Movasati \cite[Theorem 2]{GMCD-NL}, who proved it for all components passing through the Fermat point. The characterization of $\Sigma_{(1,\ldots,1)}$ as the only component passing through Fermat attaining this bound was recently established by the second author \cite[Theorem 1.1]{villa2020small} for $d\neq 3,4,6$. In this article we treat the remaining cases. 

All the previously mentioned results rely on the description of the Zariski tangent space of the local Hodge loci $V_\lambda\subseteq (T,t)$, associated to some Hodge cycle $\lambda\in H^{\frac{n}{2},\frac{n}{2}}(X_t)\cap H^n(X_t,\Z)$ for $X_t=\text{Supp}(t)\subseteq\P^{n+1}$ and $t\in T$, in terms of the infinitesimal variation of Hodge structure. In practice, instead of bounding the codimension of the components of the Hodge locus, one bounds the codimension of the Zariski tangent space of all $V_\lambda$. This is the case for all the previous results of Green, Voisin, Otwinowska and Movasati. In particular, Movasati proved that if $0\in T$ corresponds to the Fermat point then the codimension of $T_0V_\lambda$ is greater or equal than ${\frac{n}{2}+d\choose d}-(\frac{n}{2}+1)^2$ for all $\lambda\in H^{\frac{n}{2},\frac{n}{2}}(X_0)\cap H^n(X_0,\Z)$ non-trivial Hodge cycles of the Fermat variety. This naturally led Movasati to conjecture that this bound is attained if and only if $\lambda$ is the class of a linear algebraic cycle $\P^\frac{n}{2}\subseteq X_0$ of the Fermat variety \cite[Conjecture 18.8]{ho13}. Our main result disproves this conjecture for $d=3,4,6$ in all dimensions, providing a complete answer to Movasati's question for the cases not covered by \cite{villa2020small}.

\begin{theorem}
\label{theo2}
For $d=3,4,6\ge 2+\frac{6}{n}$ and $n$ even, there are infinitely many scheme-theoretically different Hodge loci $V_\lambda$ associated to non-trivial Hodge cycles of the Fermat variety $\lambda\in H^{\frac{n}{2},\frac{n}{2}}(X_0)\cap H^n(X_0,\Z)$ such that 
$$
\text{codim }T_0V_\lambda={\frac{n}{2}+d\choose d}-\left(\frac{n}{2}+1\right)^2.
$$
In particular, infinitely many of these Hodge cycles are not linear cycles. We call them \textit{fake linear cycles}. All fake linear cycles are of the form
$$
\lambda_\prim=\res\left(\frac{P_\lambda\Omega}{F^{\frac{n}{2}+1}}\right)
$$
where $P_\lambda$ is given (up to some relabeling of the coordinates) by
\begin{equation}
\label{plambda}    
P_\lambda=c_\lambda\prod_{j=1}^{\frac{n}{2}+1}\left(\frac{x_{{2j-2}}^{d-1}-(c_{2j-2}x_{{2j-1}})^{d-1}}{x_{{2j-2}}-c_{2j-2}x_{{2j-1}}}\right),
\end{equation}
where $c_0,c_2,\ldots,c_n\in\zeta_{2d}^{-3}\cdot\mathbb{S}^1_{\Q(\zeta_d)}=\{\zeta_{2d}^{-3}\cdot z\in \Q(\zeta_{2d}): \ z\in\Q(\zeta_d)\text{ and }|z|=1\}$ but not all being $d$-th roots of $-1$ simultaneously, and $c_\lambda\in \Q(\zeta_{2d})^\times$. Moreover, for any such choice of $c_i$'s, there exists some $c_\lambda\in\Q(\zeta_{2d})^\times$ such that the class $\lambda_\prim$, given by $P_\lambda$ as in \eqref{plambda}, is the class of a fake linear cycle.
\end{theorem}

We point out that the condition on the $c_i$'s not all being $d$-th roots of $-1$ simultaneously is to avoid that $\lambda_\prim$ becomes the class of a true linear cycle. Since the Hodge conjecture is known for these Fermat varieties \cite{shioda1979hodge} we know that fake linear cycles are rational combinations of linear cycles. The proof of the above result follows after a first order analysis of the Hodge loci.

Curiously Fermat varieties of degrees $d=3,4,6$ correspond exactly to those where the group $H^n(X^n_d,\Z)_\text{alg}$ of algebraic cycles has maximal rank $h^{\frac{n}{2},\frac{n}{2}}$ (see \cref{proppicmax} and \cite{beauville2014some} for a survey on these rare to find varieties). The subtle part of the above result is showing the existence of $c_\lambda$ in such a way that the corresponding class is a Hodge class. For this is necessary to describe the Galois action of $\Q(\zeta_{2d})/\Q$ on the space of totally decomposable Hodge monomials in the sense of Shioda \cite{shioda1979hodge}.
An immediate consequence of \cref{theo2} is that the Artinian Gorenstein ideal associated to each fake linear cycle is of the form
\begin{equation}
\label{idealfake}    
J^{F,\lambda}=\langle x_0-c_0x_2,\ldots,x_n-c_nx_{n+1},x_0^{d-1},\ldots,x_{n+1}^{d-1}\rangle.    
\end{equation}
The name fake linear cycle is inspired from this fact and the principle introduced by Movasati-Sert\"oz \cite{movasati2020reconstructing} which predicts that for ``good enough" algebraic cycles, one should obtain the supporting equations of a representative of the cycle as generators of $J^{F,\lambda}$ for small degrees. It was proved by Cifani-Pirola-Schlesinger \cite{cifani2021reconstructing} that all arithmetically Cohen-Macaulay curves inside a smooth surface in $\P^3$ satisfy this principle, which says that the curve can be \textit{reconstructed from its periods}. It was also shown by them that not all curves can be reconstructed from their periods (e.g. a rational degree 4 curve inside a quartic). After \eqref{idealfake} we see that fake linear cycles provide more examples (of any dimension) of algebraic cycles which cannot be reconstructed from their periods. In fact, otherwise the supporting equations of the cycle should be the $\frac{n}{2}+1$ equations of degree 1 which define a $\frac{n}{2}$-dimensional linear subvariety inside $\P^{n+1}$, but this linear variety is never contained in $X^n_d$. 

Beside the above anomalous properties of fake linear cycles, we show that their associated Hodge loci are non-reduced, completing thus the proof of following result.

\begin{theorem}
\label{theo1}
For $n$ even and $d\ge 2+\frac{6}{n}$ the unique component of minimal codimension of the Hodge locus $\text{HL}_{n,d}$ passing through the Fermat variety is $\Sigma_{(1,\ldots,1)}$, i.e. the one parametrizing hypersurfaces containing linear subvarieties of dimension $\frac{n}{2}$.
\end{theorem}

For the proof of \cref{theo1} it is necessary to compute the quadratic fundamental form of the Hodge loci associated to fake linear cycles. For this we rely on the description of this second order invariant of the IVHS introduced by Maclean \cite[Theorem 7]{maclean2005second}.

The text is organized as follows: In \S \ref{sec2} we recall the cohomology and homology of Fermat varieties. Section \S \ref{sec3} is devoted to the computation of the field of definition of totally decomposable Hodge monomials, together with the explicit description of the Galois action on them (see \cref{galact}). In \S \ref{sec4} we recall the basic results and notations about the Artinian Gorenstein ideal associated to a Hodge cycle based on \cite{villa2020small}. The proof of \cref{theo2} is given in \S \ref{sec5}. And section \S \ref{sec6} is devoted to the computation of the quadratic fundamental form associated to each fake linear cycle and the proof of \cref{theo1}.

\section{Topology of Fermat varieties}
\label{sec2}
In this section we describe the homology and cohomology groups of Fermat varieties. For this we start recalling the notations and main results of Shioda \cite{shioda1979hodge}. Let 
$$
X^n_d:=\{F:=x_0^d+\cdots +x_{n+1}^d=0\}
$$
be the $n$-dimensional Fermat variety of degree $d$. Shioda described the cohomology groups $H^n_\dR(X^n_d)$ in terms of a spectral decomposition compatible with the Hodge decomposition. This decomposition goes as follows. Let 
$$
G^n_d:=(\mu_d)^{n+2}/\Delta(\mu_d) \ , \ \ \ \ \mu_d:=\langle\zeta_d\rangle\simeq \Z/d\Z \ \text{the group of $d$-th roots of unity}.
$$
The above group acts on $X^n_d$ by coordinate-wise multiplication
\begin{equation}
\label{actionGnm}
g=(g_0,\ldots,g_{n+1}) \ , \ \ \ g\cdot x=(g_0\cdot x_0:\cdots: g_{n+1}\cdot x_{n+1}).
\end{equation}
The dual group $\hat{G}^n_d$ corresponds to the group of characters
$$
\hat{G}^n_d:=\{\alpha=(a_0,\ldots,a_{n+1})\in (\Z/d\Z)^{n+2}: a_0+\cdots+a_{n+1}=0\}
$$
whose pairing with $G^n_d$ is
$$
\alpha(g):=g_0^{a_0}\cdots g_{n+1}^{a_{n+1}}.
$$
The action of $G^n_d$ on $X^n_d$ induces an action of $G^n_d$ on $H^n(X^n_d,\Z)$ and $H^n(X^n_d,\Z)_\prim$, which naturally extends to $H^n(X^n_d,\Z)_\prim\otimes \C\simeq H^n_\dR(X^n_d)_\prim$. We have the following decomposition
\begin{equation}
\label{shiodadesc}
H^n_\dR(X^n_d)_\prim=\bigoplus_{\alpha\in \hat{G}^n_d}V(\alpha)    
\end{equation}
which is finer than the Hodge decomposition, and where 
$$
V(\alpha):=\{\omega\in H^n_\dR(X^n_d)_\prim: g^*\omega=\alpha(g)\omega \ , \ \forall g\in G^n_d\}.
$$
The following is the main result of \cite{shioda1979hodge}.

\begin{theorem}[Shioda]
\label{theoshioda}
\begin{itemize}
    \item[(i)] $\dim V(\alpha)=1$ if $a_0\cdots a_{n+1}\neq 0$, and $V(\alpha)=0$ otherwise.
    \item[(ii)] Each piece of the Hodge decomposition corresponds to
    $$
    H^{p,q}(X^n_d)_\prim=\bigoplus_{|\alpha|=q+1} V(\alpha),
    $$
    where $|\alpha|:=\frac{1}{d}\sum_{i=0}^{n+1}\overline{a_i}$, and $\overline{a_i}\in\{0,\ldots,d-1\}$ is the residue of $a_i$ modulo $d$.
    \item[(iii)] If $n$ is even, then
    $$
    (H^{\frac{n}{2},\frac{n}{2}}(X^n_d)_\prim\cap H^n(X^n_d,\Z))\otimes\C=\bigoplus_{\alpha\in \mathcal{B}^n_d}V(\alpha)
    $$
    with 
    $$
    \mathcal{B}^n_d:=\left\{\alpha\in \hat{G}^n_d: |t\cdot \alpha|=\frac{n}{2}+1 \ , \ \ \forall t\in (\Z/d\Z)^\times\right\}.
    $$
\end{itemize}
\end{theorem}

The previous result can be complemented with Griffiths basis theorem \cite{gr69}. This theorem describes the primitive cohomology classes of any smooth hypersurface $X=\{F=0\}\subseteq\P^{n+1}$ in terms of the Jacobian ring $R^F:=\C[x_0,\ldots,x_{n+1}]/J^F$, where $J^F:=\langle\frac{\partial F}{\partial x_0},\ldots,\frac{\partial F}{\partial x_{n+1}}\rangle$ is the Jacobian ideal. This description is compatible with the Hodge filtration and is done via the residue map as follows
$$
R^F_{d(q+1)-n-2}\xrightarrow{\sim}F^pH^n_\dR(X)_\prim/F^{p+1}H^n_\dR(X)_\prim
$$
$$
P\mapsto \omega_P:=\res\left(\frac{P\Omega}{F^{q+1}}\right).
$$
In the particular case of the Fermat variety one has
\begin{equation}
\label{griffbasdesc}
H^n_\dR(X^n_d)_\prim=\bigoplus_\beta\C\cdot\omega_\beta
\end{equation}
where $\omega_\beta=\res\left(\frac{x^\beta\Omega}{F^{\frac{n}{2}+1}}\right)$ and $\beta=(\beta_0,\ldots,\beta_{n+1})$ with
$\beta_i\in\{0,\ldots,d-2\}$ such that $\frac{1}{d}(\deg(x^\beta)+n+2)\in\Z$. The relation between Griffiths decomposition \eqref{griffbasdesc} and Shioda's decomposition \eqref{shiodadesc} is clarified by the following proposition.

\begin{proposition}
\label{proprelGrSh}
Let $\alpha=(a_0,\ldots,a_{n+1})\in \hat{G}^n_d$ be such that $a_0\cdots a_{n+1}\neq 0$, then 
$$
V(\alpha)=\C\cdot\omega_\beta
$$
where $\beta_i=\overline{a_i}-1$ for all $i=0,\ldots,n+1$. In particular for any polynomial $P\in R^F_{(d-2)(\frac{n}{2}+1)}$
\begin{center}
$\omega_P\in (H^{\frac{n}{2},\frac{n}{2}}(X^n_d)_\prim\cap H^n(X^n_d,\Z))\otimes\C \ \ \ $ if and only if $ \ \ \ P\in\displaystyle\bigoplus_{\begin{smallmatrix}\alpha\in\mathcal{B}^n_d, \\ V(\alpha)=\C\cdot\omega_\beta\end{smallmatrix}}\C\cdot x^\beta$.    
\end{center}
\end{proposition}

\begin{proof}
By item (i) of \cref{theoshioda} it is enough to show that $\omega_\beta\in V(\alpha)$ for $\alpha=(a_0,\ldots,a_{n+1})$ with $a_0\cdots a_{n+1}\neq 0$ and $\beta_i=\overline{a_i}-1$. Let $g=(\zeta_d^{c_0},\ldots,\zeta_d^{c_{n+1}})\in G^n_d$, then
$$
g^*\omega_\beta=\zeta_d^{\sum_{j=0}^{n+1}(\beta_j+1)c_j}\omega_\beta=\zeta_d^{\sum_{j=0}^{n+1}a_jc_j}\omega_\beta=\alpha(g)\omega_\beta.
$$
\end{proof}

\begin{remark} Note that the forms $\omega_\beta\in V(\alpha)$ for $\alpha\in\mathcal{B}^n_d$ are not Hodge cycles. In general one can show that $\omega_\beta\in H^{\frac{n}{2},\frac{n}{2}}(X^n_d)_\prim\cap H^n(X^n_d,\overline{\Q})$ assuming the Hodge conjecture.
\end{remark}

\begin{remark}
As a consequence of \cref{theoshioda}, one can show the Hodge conjecture for lots of Fermat varieties \cite{shioda1979hodge} including those of degree $d=3,4,6$. Moreover, by an elementary argument one can characterize these Fermat varieties as those where the group $H^n(X^n_d,\Z)_\text{alg}$ of algebraic cycles has maximal rank $h^{\frac{n}{2},\frac{n}{2}}$. Part of this was already noted by Beauville \cite[Proposition 11]{beauville2014some} and by Movasati \cite[Corollary 15.1]{ho13}. For the sake of completeness we will provide the argument here, starting with an elementary number theory fact which will be also used later in \cref{propvillasmall}.
\end{remark}

\begin{lemma}
\label{lemmaent}
Let $d\ge 5$ and $d\neq 6$ be a integer. Consider
$q:=\min\{p \text{ prime}: \ p\nmid 2d \}$. Then $q<\frac{d}{2}$ or $q=\frac{d+1}{2}$. The second case only holds for $d=5,9$. 
\end{lemma}

\begin{proof}
If $d=4k,$ then $\gcd(2d,\frac{d}{2}-1)=1,$ and therefore every prime $p|\frac{d}{2}-1$ satisfies that $p\nmid 2d$ and $p<\frac{d}{2}.$ Similarly, if $d=4k+2$, then $\gcd(2d,\frac{d}{2}-2)=1$ and we can take $p|\frac{d}{2}-2$. If $d=4k+3$, then $\gcd(2d,\frac{d-1}{2})=1$ and we can take $p|\frac{d-1}{2}$. If $d=4k+1$, then $\gcd(2d,\frac{d+1}{2})=1$ and so taking $p|\frac{d+1}{2}$ we conclude that $q\leq \frac{d+1}{2}$, i.e. $q\le \frac{d+1}{2}-1<\frac{d}{2}$ unless $q=\frac{d+1}{2}$. To see that this only happens for $d=5,9$ note that if $q=p_n$ is the $n$-th prime number, then $p_2\cdots p_{n-1}\mid d=2p_n-1$. One sees that $p_2\cdots p_{n-1}$ quickly becomes bigger than $2p_n-1$ for $n\ge 4$.
\end{proof}

\begin{proposition}
\label{proppicmax}
For even dimensional Fermat varieties $X^n_d$ one has
$$
\text{rank }H^n(X^n_d,\Z)_\text{alg}=h^{\frac{n}{2},\frac{n}{2}} \ \ \text{ if and only if } \ \ \varphi (d)\le 2,\text{ i.e. }d=1,2,3,4,6.
$$
\end{proposition}

\begin{proof}
Let us note first that if $\varphi(d)\le 2$, we know the Hodge conjecture by \cite{shioda1979hodge} and so it is enough to show, by \cref{theoshioda} (iii), that for all $\alpha\in \hat{G}^n_d$ with $|\alpha|=\frac{n}{2}+1$ one has
\begin{equation}
\label{eqpicmax}
|t\cdot\alpha|=\frac{n}{2}+1 \hspace{1cm}\forall t\in(\Z/d\Z)^\times.    
\end{equation}
This is trivial if $\varphi(d)=1$, and for $\varphi(d)=2$ we have $(\Z/d\Z)^\times=\{1,d-1\}$ where the result is also clear. Conversely, if $\varphi(d)>2$ let us construct some $\alpha\in \hat{G}^n_d$ with $|\alpha|=\frac{n}{2}+1$ not satisfying \eqref{eqpicmax}. Note that if we find such an $\alpha$ for $n=2$, then to construct one for any $n\ge 4$ is easy just adding pairs of entries of the form $(1,d-1)$. Thus we are reduced to the case $n=2$. Let us consider first the case $d\neq 5,9$. By \cref{lemmaent} there exists some $k\in\{2,3,\ldots, d-1\}$ such that 
$$
\frac{d}{k+1}<q<\frac{d}{k} \ \ \text{ where } \ \ q:=\min\{p \text{ prime}: \ p\nmid 2d \}.
$$
We claim the desired character is any
$$
\alpha=(aq,bq,cq,2d-(k+1)q)
$$
such that $a+b+c=k+1$ with $a,b,c\in\{1,2,\ldots,k\}$. In fact, $|\alpha|=2$ but if $t=q^{-1}\in (\Z/d\Z)^\times$ then
$$
|t\cdot\alpha|=|(a,b,c,r)|=\frac{k+1+r}{d}<2.
$$
Finally for the cases $d=5,9$ consider the characters $\alpha=(2,2,2,4), (5,5,5,3)$ respectively and $t=2$.
\end{proof}

Let us turn now to the homology groups of Fermat varieties. For this let us denote
$$
U^n_d:=\{(x_1,\ldots,x_{n+1})\in \C^{n+1} : \ 1+x_1^d+\cdots+x_{n+1}^d=0\}=X^n_d\cap \C^{n+1}
$$
the affine Fermat variety. A basis for $H_n(U^n_d,\Z)$ is given by the so called \textit{vanishing cycles}.

\begin{definition}
For every $\beta\in\{0,\ldots,d-2\}^{n+1}$ consider the homological cycle
$$
\delta_{\beta}:=\sum_{a\in\{0,1\}^{n+1}}(-1)^{\sum_{i=1}^{n+1}(1-a_i)}\Delta_{\beta+a}
$$
where $\Delta_{\beta+a}:\Delta^n:=\{(t_1,\ldots,t_{n+1})\in\R^{n+1}: t_i\ge 0 \ , \ \sum_{i=1}^{n+1}t_i=1\}\rightarrow U^n_d$ is given by
$$
\Delta_{\beta+a}(t):=\left(\zeta_{2d}^{2(\beta_1+a_1)-1}t_1^\frac{1}{d},\zeta_{2d}^{2(\beta_2+a_2)-1}t_2^\frac{1}{d},\ldots,\zeta_{2d}^{2(\beta_{n+1}+a_{n+1})-1}t_{n+1}^\frac{1}{d}\right).
$$
\end{definition}

\begin{proposition}
The set $\{\delta_\beta\}_{\beta\in\{0,\ldots,d-2\}^{n+1}}$ is a basis of $H_n(U^n_d,\Z)$.
\end{proposition}

\begin{proof}
This is a well-known fact. For a proof see for instance \cite[Remark 7.1]{ho13}.
\end{proof}

Using the Leray-Thom-Gysin sequence in homology \cite[\S 4.6]{ho13} is easy to see that
\begin{equation}
\label{lefsdescrathom}    H_n(X^n_d,\Q)=\text{Im}(H_n(U^n_d,\Q)\rightarrow H_n(X^n_d,\Q))\oplus\Q\cdot[\P^{\frac{n}{2}+1}\cap X^n_d].
\end{equation}
Hence every $\omega\in H^n_\dR(X^n_d)$ is determined by its periods over the vanishing cycles and $[\P^{\frac{n}{2}+1}\cap X^n_d]$. Since this last period is zero when $\omega\in H^n_\dR(X^n_d)_\prim$, we see that every primitive class is determined by its periods over all vanishing cycles. These periods can be explicitly computed following \cite{dmos} (see \cref{periods}).

\section{Galois action in cohomology}
\label{sec3}
Let $X\subseteq \P^{n+1}$ be a smooth hypersurface of the projective space.

\begin{definition}
\label{fielddef}
For every $\omega\in H^n_\dR(X)$, the \textit{field of definition} of $\omega$ is
$$
\Q_\omega:=\Q\left(\frac{1}{(2\pi i)^\frac{n}{2}}\int_\delta\omega \ : \ \delta\in H_n(X,\Z)\right).
$$
Since $H_n(X,\Z)$ is finitely generated, $\Q_\omega$ is also finitely generated. This is the field of definition of $\omega$ in the following sense
$$
\omega\in H^n(X,\Q_\omega).
$$
\end{definition}

\begin{definition}
For every $t\in \text{Gal}(\Q_\omega/\Q)$ we define the \textit{Galois action in cohomology} as $t(\omega)\in H^n(X,\Q_\omega)$ such that $$ 
t\left(\frac{1}{(2\pi i )^\frac{n}{2}}\int_\delta \omega\right)=\frac{1}{(2\pi i)^\frac{n}{2}}\int_\delta t(\omega) \ , \ \ \ \forall \delta\in H_n(X,\Z). 
$$ 
\end{definition}

In order to describe the Galois action in the cohomology of Fermat varieties we will use the following elementary result about periods, whose proof can be found in \cite[Lemma 7.12]{dmos} and \cite[Proposition 15.1]{ho13}.

\begin{proposition}
\label{periods}
For a Fermat variety of degree $d$ and even dimension $n$, let $\omega_\beta\in H^{\frac{n}{2},\frac{n}{2}}(X^n_d)_\prim$ and $\beta'\in\{0,\ldots,d-2\}^{n+1}$. Then
$$
\int_{\delta_{\beta'}}\omega_\beta=\frac{1}{d^{n+1}\frac{n}{2}!(2\pi i)}\prod_{i=0}^{n+1}(\zeta_d^{(\beta_i+1)(\beta'_i+1)}-\zeta_d^{(\beta_i+1)\beta'_i})\Gamma\left(\frac{\beta_i+1}{ d}\right)
$$
where $\beta_0':=0$ and $\Gamma$ is the classical Gamma function.
\end{proposition}

Using the above formula one can obtain the following elementary result which can also be found as part of \cite[Theorem 7.15]{dmos}.

\begin{proposition}
\label{propomegaprima}
For every character $\alpha=(a_0,\ldots,a_{n+1})$ with $a_0\cdots a_{n+1}\neq 0$,
$$
V(\alpha)\cap H^n(X^n_d,\Q(\zeta_d))\neq 0.
$$
In fact a generator is
$$
\eta_\alpha:=(2\pi i)^{\frac{n}{2}+1}\frac{\omega_\beta}{\prod_{i=0}^{n+1}\Gamma\left(\frac{a_i}{d}\right)}\in H^n(X^n_d,\Q(\zeta_d))_\prim,
$$
for $\beta_i=a_i-1$. Moreover, for every $t\in (\Z/d\Z)^\times\simeq\text{Gal}(\Q(\zeta_d)/\Q)$  $$ t(\eta_\alpha)=\eta_{t\cdot\alpha}. 
$$ 
\end{proposition}

\begin{proof}
This follows directly from the definition of the action, \cref{periods} and \cref{theoshioda}.
\end{proof}

\begin{definition}
We say that a character $\alpha\in \mathcal{B}^n_d$ is \textit{totally decomposable} if we can relabel the entries of $\alpha$ in such a way that 
\begin{equation}
\label{totdecchar}    
\alpha=(a_0,d-a_0,a_2,d-a_2,\ldots,a_n,d-a_n).
\end{equation}

\begin{remark}
The polynomial $P_\lambda$ given by \eqref{plambda} is a $\C$-linear combination of the monomials $x^\beta$  with $\beta_{2j-2}+\beta_{2j-1}=d-2$  for $j=1,\dots,\frac{n}{2}+1.$ Each of these $\beta$'s has an associated character $\alpha\in \mathcal{B}^n_d$ that is totally decomposable with $a_j=\beta_j+1.$ In the following proposition we restrict the field of definition of $\omega_\beta=\res\left(\frac{x^\beta\Omega}{F^{\frac{n}{2}+1}}\right)$ where $\beta$ has associated character $\alpha$ totally decomposable.     
\end{remark}

\end{definition}

\begin{proposition}
\label{galact}
For every $\alpha=(a_0,a_1,\ldots,a_n,a_{n+1})\in\mathcal{B}^n_d$ totally decomposable of the form \eqref{totdecchar}, and $\beta_i=a_i-1$, 
$$
\Q_{\omega_\beta}\subseteq\Q(\zeta_{2d}).
$$
Furthermore for every $t\in\Gal(\Q(\zeta_{2d})/\Q)\simeq(\Z/2d\Z)^\times$
$$
t(\omega_\beta)=(-1)^\frac{\sum_{j=1}^{\frac{n}{2}+1}(ta_{2j-2}-\overline{ta_{2j-2}})}{d}\omega_\gamma,
$$
where $\omega_\gamma\in V(t\cdot\alpha)$ and $\overline{a}$ denotes the residue of $a\in\Z$ modulo $d$. 
\end{proposition}

\begin{proof}
Consider the class of the linear cycle $\P^\frac{n}{2}=\{x_0-\zeta_{2d}x_1=\cdots=x_n-\zeta_{2d}x_{n+1}=0\}$. Then by \cite[Theorem 1.1]{villaflor2021periods} and \cref{theoshioda} we know that
$$
\omega_P=\frac{-1}{\frac{n}{2}!\cdot d^\frac{n}{2}}[\P^\frac{n}{2}]_\prim\in H^{\frac{n}{2},\frac{n}{2}}(X^n_d)_\prim\cap H^n(X^n_d,\Q),
$$
where
$$
P=\zeta_{2d}^{\frac{n}{2}+1}\sum_{\beta\in I}x^\beta\zeta_{2d}^{\beta_1+\beta_3+\cdots+\beta_{n+1}}
$$
and 
$$
I:=\left\{(\beta_0,\ldots,\beta_{n+1})\in\{0,\ldots,d-2\}^{n+2}: \beta_{2j-2}+\beta_{2j-1}=d-2 \ , \ \forall j=1,\ldots,\frac{n}{2}+1\right\}.
$$
Let us first show that $\Q_{\omega_\beta}\subseteq\Q(\zeta_{2d})$. Since $\Q_{\eta_\alpha}\subseteq\Q(\zeta_d)$ it is enough to show that
$$
C_\beta:=\frac{\prod_{i=0}^{n+1}\Gamma\left(\frac{a_i}{d}\right)}{(2\pi i)^{\frac{n}{2}+1}}\in\Q(\zeta_{2d}).
$$
This could be shown directly by using the properties of the Gamma function, but we will give another proof. Let $K/\Q(\zeta_{2d})$ be a Galois extension such that $C_\beta\in K$. For any $\sigma\in \Gal(K/\Q(\zeta_{2d}))$ we have $\sigma(\omega_P)=\omega_P$, since it is a rational class. Hence by \cref{propomegaprima}
$$
\sum_{\beta\in I}\zeta_{2d}^{a_1+a_3+\cdots+a_{n+1}}\sigma(C_\beta)\cdot\eta_\alpha=\sum_{\beta\in I}\zeta_{2d}^{a_1+a_3+\cdots+a_{n+1}}C_\beta\cdot\eta_\alpha,
$$
in other words $\sigma(C_\beta)=C_\beta$ for all $\sigma\in\Gal(K/\Q(\zeta_{2d}))$, i.e. $C_\beta\in\Q(\zeta_{2d})$ as claimed. Let us now compute the Galois action of $\Gal(\Q(\zeta_{2d})/\Q)$ on $\omega_\beta$. Let $t\in \Gal(\Q(\zeta_{2d})/\Q)\simeq(\Z/2d\Z)^\times$, then again $t(\omega_P)=\omega_P$, since $\omega_P$ is a rational class. Expanding this equality we have
$$
\sum_{\beta\in I}\zeta_{2d}^{t(a_1+a_3+\cdots+a_{n+1})}t(\omega_\beta)=\sum_{\beta\in I}\zeta_{2d}^{a_1+a_3+\cdots+a_{n+1}}\omega_\beta.
$$
Since by \cref{propomegaprima} we know $t(\omega_\beta)=C\cdot\omega_\gamma$ for some $C\in\Q(\zeta_{2d})^\times$ and $\omega_\gamma\in V(t\cdot\alpha)$, we get that
$$
\zeta_{2d}^{t(a_1+a_3+\cdots+a_{n+1})}t(\omega_\beta)=\zeta_{2d}^{\overline{ta_1}+\overline{ta_3}+\cdots+\overline{ta_{n+1}}}\omega_\gamma
$$
and the result follows. For the last equality just note that $t(\omega_\beta)=t(C_\beta)\cdot\eta_{t\cdot\alpha}$.
\end{proof}

\begin{remark}
\label{rmkcebeta}
Using Euler's reflection formula we can compute explicitly 
$$
C_\beta=\frac{\prod_{j=1}^{\frac{n}{2}+1}\Gamma\left(\frac{a_{2j-2}}{d}\right)\Gamma\left(1-\frac{a_{2j-2}}{d}\right)}{(2\pi i)^{\frac{n}{2}+1}}=\frac{\prod_{j=1}^{\frac{n}{2}+1}\frac{\pi}{\sin{(\pi a_{2j-2}/d)}}}{(2\pi i )^{\frac{n}{2}+1}}=\prod_{j=1}^{\frac{n}{2}+1}\frac{1}{\zeta_{2d}^{a_{2j-2}}-\zeta_{2d}^{-a_{2j-2}}}.
$$
\end{remark}

\section{Artinian Gorenstein ideal associated to a Hodge cycle}
\label{sec4}
For the sake of completeness we will briefly recall some known facts about Artinian Gorenstein ideals associated to Hodge cycles in smooth hypersurfaces of the projective space. Our main aim is to settle the notation we will use in the rest of the article and to gather some facts from \cite{villa2020small}.

\begin{definition}
A graded $\C$-algebra $R$ is \textit{Artinian Gorenstein} if there exist $\sigma\in \N$ such that 
\begin{itemize}
    \item[(i)] $R_e=0\text{ for all }e>\sigma$,
    \item[(ii)] $\dim_\C \ R_\sigma=1$,
    \item[(iii)] $\text{the multiplication map }R_i\times R_{\sigma-i}\rightarrow R_\sigma\text{ is a perfect pairing for all }i=0,\ldots,\sigma.$
\end{itemize}
The number $\sigma=:\text{soc}(R)$ is the \textit{socle of $R$}. We say that an ideal $I\subseteq\C[x_0,\ldots,x_{n+1}]$ is \textit{Artinian Gorenstein of socle} $\sigma=:\text{soc}(I)$ if the quotient ring $R=\C[x_0,\ldots,x_{n+1}]/I$ is Artinian Gorenstein of socle $\sigma$.
\end{definition}

The definition of the following ideal appeared first in the work of Voisin \cite{voisin89} for surfaces, and later in the work of Otwinowska \cite{Otwinowska2003} for higher dimensional varieties.

\begin{definition}
Let $X=\{F=0\}\subseteq\P^{n+1}$ be a smooth degree $d$ hypersurface of even dimension $n$, and $\lambda\in H^{\frac{n}{2},\frac{n}{2}}(X,\Z)$ be a non-trivial Hodge cycle. Consider $J^F:=\langle \frac{\partial F}{\partial x_0},\ldots,\frac{\partial F}{\partial x_{n+1}}\rangle$ to be the Jacobian ideal, we define the \textit{Artinian Gorenstein ideal associated to $\lambda$} as
\begin{equation}
\label{ideal}    
J^{F,\lambda}:=(J^F:P_\lambda),
\end{equation}
where $P_\lambda\in \C[x_0,\ldots,x_{n+1}]_{(d-2)(\frac{n}{2}+1)}$ is such that $\lambda_\prim=\res\left(\frac{P_\lambda\Omega}{F^{\frac{n}{2}+1}}\right)^{\frac{n}{2},\frac{n}{2}}$. This ideal is Artinian Gorenstein of $\text{soc}(J^{F,\lambda})=(d-2)(\frac{n}{2}+1)=\frac{1}{2}\text{soc}(J^F)$.
\end{definition}

The importance of this ideal is due to the following proposition which relates it with to the local Hodge locus $V_\lambda$ associated to the Hodge cycle $\lambda$.

\begin{proposition}
Let $X=\{F=0\}\subseteq\P^{n+1}$ be a smooth degree $d$ hypersurface of even dimension $n$, and consider two Hodge cycles $\lambda_1,\lambda_2\in H^{\frac{n}{2},\frac{n}{2}}(X,\Z)$. Then
$$
J^{F,\lambda_1}=J^{F,\lambda_2} \ \ \Longleftrightarrow \ \ \exists c\in\Q^\times: (\lambda_1-c\cdot\lambda_2)_\prim=0 \ \ \Longleftrightarrow \ \ V_{\lambda_1}=V_{\lambda_2}.
$$
\end{proposition}

\begin{proof}
See \cite[Corollary 2.3]{villa2020small}.
\end{proof}

Moreover, this ideal encodes in a simple way the information of the first order approximation of the Hodge loci. In fact the content of the following proposition is a rephrasing of the classical result of Carlson-Green-Griffiths-Harris on the infinitesimal variation of Hodge structure for hypersurfaces \cite{gri83III}.

\begin{proposition}
\label{propAGtang}
Let $T\subseteq\C[x_0,\ldots,x_{n+1}]_d$ be the parameter space of smooth degree $d$ hypersurfaces of $\P^{n+1}$, of even dimension $n$. For $t\in T$, let $X_t=\{F=0\}\subseteq \P^{n+1}$ be the corresponding hypersurface. For every Hodge cycle $\lambda\in H^{\frac{n}{2},\frac{n}{2}}(X_t,\Z)$, we can compute the Zariski tangent space of its associated Hodge locus $V_\lambda$ as
$$
T_tV_\lambda=J^{F,\lambda}_d.
$$
Where we have identified $T_tT\simeq \C[x_0,\ldots,x_{n+1}]_d$.
\end{proposition}

\begin{proof}
See \cite[Proposition 2.1, Proposition 2.2]{villa2020small}.
\end{proof}

Using the previous result, we can obtain the following technical lemma which is the first step in the proof of \cref{theo2}.

\begin{lemma}
\label{lemmatech}
Let $X^n_d=\{F=0\}$ be the Fermat variety of even dimension $n$ and degree $d\ge 2+\frac{6}{n}$. Let $\lambda\in H^{\frac{n}{2},\frac{n}{2}}(X^n_d,\Z)$ be a non-trivial Hodge cycle such that 
$$
\text{codim }T_0V_\lambda={\frac{n}{2}+d\choose d}-\left(\frac{n}{2}+1\right)^2.
$$
Then there exist $c_\lambda,c_0,c_2,c_4,\ldots,c_n\in\C^\times$ such that up to a permutation of the coordinates $\lambda_\prim=\res\left(\frac{P_\lambda\Omega}{F^{\frac{n}{2}+1}}\right)$ where $P_\lambda$ is given by \eqref{plambda}, i.e.
$$
P_\lambda=c_\lambda\prod_{j=1}^{\frac{n}{2}+1}\left(\frac{x_{{2j-2}}^{d-1}-(c_{2j-2}x_{{2j-1}})^{d-1}}{x_{{2j-2}}-c_{2j-2}x_{{2j-1}}}\right).
$$
\end{lemma}

\begin{proof}
This follows from \cite[Proposition 5.3, Proposition 4.1]{villa2020small}. The final assertion that $\res\left(\frac{P_\lambda\Omega}{F^{\frac{n}{2}+1}}\right)=\res\left(\frac{P_\lambda\Omega}{F^{\frac{n}{2}+1}}\right)^{\frac{n}{2},\frac{n}{2}}$ follows from \cref{theoshioda} \cite[Theorem 1]{shioda1979hodge}.
\end{proof}

\section{Proof of \cref{theo2}}
\label{sec5}
In this section we will prove \cref{theo2}, thus characterizing fake linear cycles as residue forms. In order to do this we will first bound the field of definition of all fake linear cycles by computing their periods, then we will characterize them as those invariant under the Galois action.

\begin{proposition}
\label{propboundfd}
In the same context of \cref{lemmatech} we have that $c_\lambda\in \Q(\zeta_{2d})^\times$ and $c_0,c_2,\ldots,c_n\in \zeta_{2d}^{-3}\cdot\mathbb{S}^1_{\Q(\zeta_{d})}=\{\zeta_{2d}^{-3}\cdot z\in \Q(\zeta_{2d}) \ : \ z\in\Q(\zeta_d) \text{ and }|z|=1\}$. In consequence $\lambda_\prim$ is a $\Q(\zeta_{2d})$-linear combination of residue forms $\omega_\beta$ with $\Q_{\omega_\beta}\subseteq\Q(\zeta_{2d})$.
\end{proposition}

\begin{proof}
Since $\lambda_\prim=\res\left(\frac{P_\lambda\Omega}{F^{\frac{n}{2}+1}}\right)$ is a Hodge class, all its periods are rational numbers. Using the formula given by \cref{periods} together with \cref{rmkcebeta} we have that
$$
\frac{1}{(2\pi i)^\frac{n}{2}}\int_{\delta_{\beta'}}\lambda_\prim=\frac{c_\lambda}{d^{\frac{n}{2}+1}\cdot \frac{n}{2}!}\sum_{\beta\in I}\prod_{j=1}^{\frac{n}{2}+1}\frac{c_{2j-2}^{\beta_{2j-1}}}{\zeta_{2d}^{\beta_{2j-2}+1}-\zeta_{2d}^{-\beta_{2j-2}-1}}\prod_{i=0}^{n+1}(\zeta_d^{(\beta_i+1)(\beta_i'+1)}-\zeta_d^{(\beta_i+1)\beta_i'})
$$
$$
=\frac{c_\lambda}{d^{\frac{n}{2}+1}\cdot \frac{n}{2}!}\sum_{\beta\in I}c_0^{\beta_1}\cdot c_2^{\beta_3}\cdots c_n^{\beta_{n+1}}\zeta_{2d}^{\beta_0+\beta_2+\cdots+\beta_n+\frac{n}{2}+1}\cdot \zeta_d^{\sum_{i=0}^{n+1}(\beta_i+1)\beta_i'}\frac{\prod_{i=0}^{n+1}(\zeta_d^{\beta_i+1}-1)}{\prod_{j=1}^{\frac{n}{2}+1}(\zeta_d^{\beta_{2j-2}+1}-1)}
$$
$$
=\frac{c_\lambda(c_0c_2\cdots c_n)^{-1}}{d^{\frac{n}{2}+1}\cdot \frac{n}{2}!}\sum_{\beta_1,\beta_3,\ldots,\beta_{n+1}=0}^{d-2}\prod_{j=1}^{\frac{n}{2}+1}(\zeta_{2d}^{-1}c_{2j-2}\zeta_d^{\beta_{2j-1}'-\beta_{2j-2}'})^{\beta_{2j-1}+1}(1-\zeta_d^{\beta_{2j-1}+1})
$$
$$
=\frac{c_\lambda(c_0c_2\cdots c_n)^{-1}}{d^{\frac{n}{2}+1}\cdot \frac{n}{2}!}\prod_{j=1}^{\frac{n}{2}+1}\left(\sum_{\ell=1}^{d-1}(c_{2j-2}\zeta_{2d}^{2(\beta_{2j-1}'-\beta_{2j-2}')-1})^\ell-(c_{2j-2}\zeta_{2d}^{2(\beta_{2j-1}'-\beta_{2j-2}')+1})^\ell\right)
$$
$$
=\frac{c_\lambda(c_0c_2\cdots c_n)^{-1}}{d^{\frac{n}{2}+1}\cdot \frac{n}{2}!}\prod_{j=1}^{\frac{n}{2}+1} E_{j,\beta'}\in\Q \ \ \ , \ \ \ \forall \beta'\in \{0,1,\ldots,d-2\}^{n+1},
$$
where each $E_{j,\beta'}=\sum_{\ell=1}^{d-1}(c_{2j-2}\zeta_{2d}^{2(\beta_{2j-1}'-\beta_{2j-2}')-1})^\ell-(c_{2j-2}\zeta_{2d}^{2(\beta_{2j-1}'-\beta_{2j-2}')+1})^\ell$. If $c_{2j-2}^d=-1$, we can always choose some $\beta_{2j-1}',\beta_{2j-2}'\in\{0,1,\ldots,d-2\}$ such that $E_{j,\beta'}\neq 0$. Let us denote 
$$
S:=\{j\in\{1,2,\ldots,\frac{n}{2}+1\}: c_{2j-2}^d=-1\}
$$
and consider the set $\mathfrak{B}$ of all $\beta'\in\{0,1,\ldots,d-2\}^{n+1}$ such that the value of $E_{j,\beta'}\neq 0$ is fixed for every $j\in S$. For every $\beta'\in \mathfrak{B}$ we have that for $j\notin S$
$$
E_{j,\beta'}=\frac{c_{2j-2}(c_{2j-2}^d+1)\zeta_{2d}^{2(\beta_{2j-1}'-\beta_{2j-2}')-1}(1-\zeta_d)}{(c_{2j-2}\cdot \zeta_{2d}^{2(\beta_{2j-1}'-\beta_{2j-2}')-1}-1)(c_{2j-2}\cdot \zeta_{2d}^{2(\beta_{2j-1}'-\beta_{2j-2}')+1}-1)}\neq 0.
$$
It is clear that $c_{2j-2}\in\zeta_{2d}^{-3}\cdot\mathbb{S}^1_{\Q(\zeta_{d})}$ for $j\in S$. In order to show that $c_{2j-2}\in\zeta_{2d}^{-3}\cdot\mathbb{S}^1_{\Q(\zeta_{d})}$ for $j\notin S$, fix some $j_0\notin S$ and consider two $\beta',\beta''\in\mathfrak{B}$ such that $E_{j,\beta'}=E_{j,\beta''}$ for all $j\neq j_0$ and $\beta_{2j_0-1}'-\beta_{2j_0-2}'=1$, $\beta_{2j_0-1}''-\beta_{2j_0-2}''=0$. Then
$$
\frac{\int_{\delta_{\beta'}}\lambda_\prim}{\int_{\delta_{\beta''}}\lambda_\prim}=\frac{E_{j_0,\beta'}}{E_{j_0,\beta''}}=\frac{\zeta_d(c_{2j_0-2}\cdot \zeta_{2d}^{-1}-1)}{c_{2j_0-2}\cdot \zeta_{2d}^3-1}=q\in\Q^\times
$$
and so 
$$
c_{2j_0-2}=\frac{q- \zeta_d}{\zeta_{2d}^3(q-\zeta_{d}^{-1})}\in\zeta_{2d}^{-3}\cdot \mathbb{S}^1_{\Q(\zeta_{d})}.
$$
Finally, since for every $\beta'\in\mathfrak{B}$ we know that $E_{j,\beta'}\in\Q(\zeta_{2d})$, it follows from the above formula for $\frac{1}{(2\pi i)^{\frac{n}{2}}}\int_{\delta_{\beta'}}\lambda_\prim\in\Q$ that $c_\lambda\in\Q(\zeta_{2d})$.
\end{proof}

\begin{remark}
\label{rmk2}
 Note that after \cref{lemmatech} and \cref{propboundfd} we know that all fake linear cycles are of the form
$$
\lambda_\prim=\res\left(\frac{P_\lambda\Omega}{F^{\frac{n}{2}+1}}\right)
$$
for $P_\lambda$ given by \eqref{plambda} where $c_\lambda\in \Q(\zeta_{2d})^\times$ and $c_0,c_2,\ldots,c_n\in \zeta_{2d}^{-3}\cdot \mathbb{S}^1_{\Q(\zeta_d)}$. In order to complete the proof of \cref{theo2} we only need to prove that for any choice of $c_0,c_2,\ldots,c_n\in \zeta_{2d}^{-3}\cdot\mathbb{S}^1_{\Q(\zeta_d)}$ there exists some $c_\lambda\in \Q(\zeta_{2d})^\times$ such that $\lambda$ is in fact a Hodge class, i.e. such that 
$$
\Q_\lambda=\Q.
$$
In terms of Galois cohomology, to prove the existence of such $c_\lambda$, is equivalent to find a number $c_\lambda\in\Q(\zeta_{2d})^\times$ such that 
$$
\sigma(\lambda)=\lambda
$$
for all $\sigma\in\Gal(\Q(\zeta_{2d})/\Q)$. This in turn translates into collection of relations of the form 
$$
\sigma(c_\lambda)=c_\lambda\cdot \phi_\sigma
$$
for some numbers $\phi_\sigma(c_0,c_2,\ldots,c_n)\in\Q(\zeta_{2d})^\times$ which can be explicitly computed case by case. Since the set $\{{\sigma(c_\lambda)}/{c_\lambda}\}$ is by definition a 1-coboundary in the group cohomology of $G=\Gal(\Q(\zeta_{2d})/\Q)$ with coefficients in $\Q(\zeta_{2d})^\times$, the theorem will follow if we show that $\{\phi_\sigma\}$ is a 1-cocycle by the following well known result which can be found in \cite{neukirch2013cohomology}.
\end{remark}

\begin{theorem}[Hilbert's Theorem 90]
If $L/K$ is a finite Galois extension of fields with Galois group $G=\Gal(L/K)$, then the first group cohomology $H^1(G,L^\times)=\{1\}$.
\end{theorem}

Now we are in position to prove \cref{theo2}, but we will divide the proof into the three possible cases $d=3,4,6$. Along all the proofs we will denote by 
$$
I:=\left\{(\beta_0,\ldots,\beta_{n+1})\in\{0,\ldots,d-2\}^{n+2}: \beta_{2j-2}+\beta_{2j-1}=d-2 \ , \ \forall j=1,\ldots,\frac{n}{2}+1\right\}
$$
the set of multi-indexes corresponding to the monomials of $P_\lambda$ given by \eqref{plambda}.

\begin{theorem}
For the Fermat cubic $X^n_3$ with $n\ge 6$, all fake linear cycles are of the form
$$
\lambda_\prim=\res\left(\frac{P_\lambda\Omega}{F^{\frac{n}{2}+1}}\right)
$$
for $P_\lambda$ given by \eqref{plambda}, where $c_0,c_2,\ldots,c_n\in\mathbb{S}^1_{\Q(\zeta_6)}$ but not all $3$-th roots of $-1$ simultaneously, and $c_\lambda\in \Q(\zeta_{6})^\times$. Moreover, for any such choice of $c_i$'s, there exists some $c_\lambda\in\Q(\zeta_{6})^\times$ such that the class $\lambda_\prim$, given by $P_\lambda$ as in \eqref{plambda}, is the class of a fake linear cycle.
\end{theorem}

\begin{proof}
Since all the monomials of $P_\lambda$ are totally decomposable, and all their accompanying coefficients belong to $\Q(\zeta_6)$ we know (by \cref{galact}) that 
$$
\Q_{\lambda}\subseteq \Q(\zeta_6).
$$
In order to show that $\Q_\lambda=\Q$ it is enough to show that $\lambda$ is invariant under the action of $\Gal(\Q(\zeta_6)/\Q)=\{id, \sigma\}$ where $\sigma(\zeta_6)=\zeta_6^{-1}=\overline{\zeta_6}$. In particular for every $\alpha\in\Q(\zeta_6)$, $\alpha=a+b\zeta_6$ for $a,b\in\Q$ and so $\sigma(\alpha)=a+b\overline{\zeta_6}=\overline{\alpha}$. With this we conclude that for  $\gamma_i=1-\beta_i$ 
$$
\sigma(\lambda)=\sigma({c}_\lambda)\sum_{\beta\in I}(-1)^{\frac{\sum_{j=1}^{\frac{n}{2}+1}(5(\beta_{2j-2}+1)-\overline{5(\beta_{2j-2}+1)})}{3}}\omega_\gamma\prod_{j=1}^{\frac{n}{2}+1}c_{2j-2}^{-\beta_{2j-1}}.
$$
Hence
$$
\sigma(\lambda)=\lambda \ \ \ \text{
if and only if } \ \ \ \frac{\sigma(c_\lambda)}{c_\lambda}=(-1)^{\frac{n}{2}+1}c_0\cdot c_2\cdots c_n.
$$
Since $N((-1)^{\frac{n}{2}+1}c_0\cdot c_2\cdots c_n)=|(-1)^{\frac{n}{2}+1}c_0\cdot c_2\cdots c_n|^2=1$, we know such $c_\lambda$ always exists by Hilbert's Theorem 90.
\end{proof}

\begin{theorem}
For the Fermat quartic $X^n_4$ with $n\ge 4$, all fake linear cycles are of the form
$$
\lambda_\prim=\res\left(\frac{P_\lambda\Omega}{F^{\frac{n}{2}+1}}\right)
$$
for $P_\lambda$ given by \eqref{plambda}, where $c_0,c_2,\ldots,c_n\in\zeta_8\cdot\mathbb{S}^1_{\Q(i)}$ but not all $4$-th roots of $-1$ simultaneously, and $c_\lambda\in \Q(\zeta_{8})^\times$. Moreover, for any such choice of $c_i$'s, there exists some $c_\lambda\in\Q(\zeta_{8})^\times$ such that the class $\lambda_\prim$, given by $P_\lambda$ as in \eqref{plambda}, is the class of a fake linear cycle.
\end{theorem}

\begin{proof}
Note first that $\zeta_8^{-3}\cdot\mathbb{S}^1_{\Q(\zeta_4)}=\zeta_8\cdot \mathbb{S}^1_{\Q(i)}$. Since all the monomials of $P_\lambda$ are totally decomposable, and all their accompanying coefficients belong to $\Q(\zeta_8)$ we see that 
$$
\Q_{\lambda}\subseteq \Q(\zeta_8).
$$
In order to show that $\Q_\lambda=\Q$ it is enough to show that $\lambda$ is invariant under the action of $\Gal(\Q(\zeta_8)/\Q)=\{id, \sigma_3,\sigma_5,\sigma_7\}$ where $\sigma_j(\zeta_8)=\zeta_8^{j}$. Observe that $\sigma_7(\zeta_8)=\zeta_8^{-1}=\overline{\zeta_8}$. In particular for every $\alpha=a\zeta_8+b\zeta_8^3\in \zeta_8\cdot \mathbb{S}^1_{\Q(i)}$ we have 
$$
\sigma_3(\alpha)=-\overline{\alpha} \ , \ \ \ 
\sigma_5(\alpha)=-\alpha \ , \ \ \
\sigma_7(\alpha)=\overline{\alpha}.
$$
With this and \cref{galact} we conclude that for $\gamma_i=2-\beta_i$
$$
\sigma_7(\lambda)=\sigma_7(c_\lambda)\sum_{\beta\in I}(-1)^{\frac{\sum_{j=1}^{\frac{n}{2}+1}(7(\beta_{2j-2}+1)-\overline{7(\beta_{2j-2}+1)})}{4}}\omega_\gamma\prod_{j=1}^{\frac{n}{2}+1}c_{2j-2}^{-\beta_{2j-1}}.
$$
Hence
\begin{equation}
\label{eq4.1}
\sigma_7(\lambda)=\lambda \ \ \ \text{
if and only if } \ \ \ \frac{\sigma_7(c_\lambda)}{c_\lambda}=(-1)^{\frac{n}{2}+1}(c_0\cdot c_2\cdots c_n)^2.
\end{equation}
On the other hand for $\gamma_i=\beta_i$
$$
\sigma_5(\lambda)=\sigma_5(c_\lambda)\sum_{\beta\in I}(-1)^{\frac{\sum_{j=1}^{\frac{n}{2}+1}(5(\beta_{2j-2}+1)-\overline{5(\beta_{2j-2}+1)})}{4}}\omega_\gamma\prod_{j=1}^{\frac{n}{2}+1}(-c_{2j-2})^{\beta_{2j-1}}.
$$
Hence
\begin{equation}
\label{eq4.2}
\sigma_5(\lambda)=\lambda \ \ \ \text{ if and only if } \ \ \ \frac{\sigma_5(c_\lambda)}{c_\lambda}=(-1)^{\frac{n}{2}+1}.
\end{equation}
Finally for $\gamma_j=2-\beta_j$
$$
\sigma_3(\lambda)=\sigma_3(c_\lambda)\sum_{\beta\in I}(-1)^{\frac{\sum_{j=1}^{\frac{n}{2}+1}(3(\beta_{2j-2}+1)-\overline{3(\beta_{2j-2}+1)})}{4}}\omega_\gamma\prod_{j=1}^{\frac{n}{2}+1}(-c_{2j-2})^{-\beta_{2j-1}}.
$$
Hence
\begin{equation}
\label{eq4.3}
\sigma_3(\lambda)=\lambda \ \ \ \text{ if and only if } \ \ \ \frac{\sigma_3(c_\lambda)}{c_\lambda}=(c_0\cdot c_2\cdots c_n)^2.    
\end{equation}
Equations \eqref{eq4.1}, \eqref{eq4.2} and \eqref{eq4.3} imply the existence of the desired $c_\lambda$ such that $\Q_\lambda=\Q$, if and only if, the map $\phi:\Gal(\Q(\zeta_8)/\Q)\rightarrow \Q(\zeta_8)^\times$ given by 
$$
\phi(id)=1,
$$
$$
\phi(\sigma_3)=(c_0\cdot c_2\cdots c_n)^2,
$$
$$
\phi(\sigma_5)=(-1)^{\frac{n}{2}+1},
$$
$$
\phi(\sigma_7)=(-1)^{\frac{n}{2}+1}(c_0\cdot c_2\cdots c_n)^2
$$
is a 1-coboundary. By Hilbert's Theorem 90 we know $H^1(G,L^\times)=\{1\}$ for $L=\Q(\zeta_8)$ and $G=\Gal(\Q(\zeta_8)/\Q)$ and so we get the existence of the desired $c_\lambda\in \Q(\zeta_8)$ after noting that $\phi$ is a 1-cocycle by definition.
\end{proof}

\begin{theorem}
For the Fermat sextic $X^n_6$ with $n\ge 2$, all fake linear cycles are of the form
$$
\lambda_\prim=\res\left(\frac{P_\lambda\Omega}{F^{\frac{n}{2}+1}}\right)
$$
for $P_\lambda$ given by \eqref{plambda}, where $c_0,c_2,\ldots,c_n\in i\cdot\mathbb{S}^1_{\Q(\zeta_6)}$ but not all $6$-th roots of $-1$ simultaneously, and $c_\lambda\in \Q(\zeta_{12})^\times$. Moreover, for any such choice of $c_i$'s, there exists some $c_\lambda\in\Q(\zeta_{12})^\times$ such that the class $\lambda_\prim$, given by $P_\lambda$ as in \eqref{plambda}, is the class of a fake linear cycle.
\end{theorem}

\begin{proof}
Note first that all the elements of $\zeta_{12}^{-3}\cdot\mathbb{S}^1_{\Q(\zeta_6)}=i\cdot \mathbb{S}^1_{\Q(\zeta_6)}$ are of the form $a\zeta_{12}+b\zeta_{12}^3$ for $a,b\in\Q$. Since all the monomials of $P_\lambda$ are totally decomposable, and all their accompanying coefficients belong to $\Q(\zeta_{12})$ we see that 
$$
\Q_{\lambda}\subseteq \Q(\zeta_{12}).
$$
In order to show that $\Q_\lambda=\Q$ it is enough to show that $\lambda$ is invariant under the action of $\Gal(\Q(\zeta_{12})/\Q)=\{id, \sigma_5,\sigma_7,\sigma_{11}\}$ where $\sigma_j(\zeta_{12})=\zeta_{12}^{j}$. Observe that $\sigma_{11}(\zeta_{12})=\zeta_{12}^{-1}=\overline{\zeta_{12}}$. In particular for every $\alpha=a\zeta_{12}+b\zeta_{12}^3$ with $a,b\in\Q,$ we have 
$$
\sigma_5(\alpha)=-\overline{\alpha} \ , \ \ \ \sigma_7(\alpha)=-\alpha \ , \ \ \ \sigma_{11}(\alpha)=\overline{\alpha}.
$$
With this and \cref{galact} we conclude that for $\gamma_i=4-\beta_i$
$$
\sigma_{11}(\lambda)=\sigma_{11}(c_\lambda)\sum_{\beta\in I}(-1)^{\frac{\sum_{j=1}^{\frac{n}{2}+1}(11(\beta_{2j-2}+1)-\overline{11(\beta_{2j-2}+1)})}{6}}\omega_\gamma\prod_{j=1}^{\frac{n}{2}+1}c_{2j-2}^{-\beta_{2j-1}}.
$$
Hence
\begin{equation}
\label{eq6.1}    
\sigma_{11}(\lambda)=\lambda \ \ \ \text{ if and only if } \ \ \ \frac{\sigma_{11}(c_\lambda)}{c_\lambda}=(-1)^{\frac{n}{2}+1}(c_0\cdot c_2\cdots c_n)^4.
\end{equation}
On the other hand for $\gamma_i=\beta_i$
$$
\sigma_7(\lambda)=\sigma_7(c_\lambda)\sum_{\beta\in I}(-1)^{\frac{\sum_{j=1}^{\frac{n}{2}+1}(7(\beta_{2j-2}+1)-\overline{7(\beta_{2j-2}+1)})}{6}}\omega_\gamma\prod_{j=1}^{\frac{n}{2}+1}(-c_{2j-2})^{\beta_{2j-1}}.
$$
Hence
\begin{equation}
\label{eq6.2}    
\sigma_7(\lambda)=\lambda \ \ \ \text{ if and only if } \ \ \ \frac{\sigma_7(c_\lambda)}{c_\lambda}=(-1)^{\frac{n}{2}+1}.
\end{equation}
Finally for $\gamma_j=4-\beta_j$
$$
\sigma_5(\lambda)=\sigma_5(c_\lambda)\sum_{\beta\in I}(-1)^{\frac{\sum_{j=1}^{\frac{n}{2}+1}(5(\beta_{2j-2}+1)-\overline{5(\beta_{2j-2}+1)})}{6}}\omega_\gamma\prod_{j=1}^{\frac{n}{2}+1}(-c_{2j-2})^{-\beta_{2j-1}}.
$$
Hence
\begin{equation}
\label{eq6.3}    
\sigma_5(\lambda)=\lambda \ \ \ \text{ if and only if } \ \ \ \frac{\sigma_5(c_\lambda)}{c_\lambda}=(c_0\cdot c_2\cdots c_n)^4.
\end{equation}
Equations \eqref{eq6.1}, \eqref{eq6.2} and \eqref{eq6.3} imply the existence of the desired $c_\lambda$, if and only if, the map $\phi:\Gal(\Q(\zeta_{12})/\Q)\rightarrow \Q(\zeta_{12})^\times$ given by 
$$
\phi(id)=1,
$$
$$
\phi(\sigma_5)=(c_0\cdot c_2\cdots c_n)^4,
$$
$$
\phi(\sigma_7)=(-1)^{\frac{n}{2}+1},
$$
$$
\phi(\sigma_{11})=(-1)^{\frac{n}{2}+1}(c_0\cdot c_2\cdots c_n)^4
$$
is a 1-coboundary. By Hilbert's Theorem 90 we know $H^1(G,L^\times)=\{1\}$ for $L=\Q(\zeta_{12})$ and $G=\Gal(\Q(\zeta_{12})/\Q)$. Thus $c_\lambda\in \Q(\zeta_{12})$ exists since $\phi$ is by definition a 1-cocycle.
\end{proof}

\begin{remark}
We want to highlight that using the Galois action in cohomology it is also possible to obtain another proof of \cite[Theorem 1.1]{villa2020small} as follows.
\end{remark}

\begin{proposition}
\label{propvillasmall}
There are no fake linear cycles inside $X^n_d$ for $d\ge 2+\frac{6}{n}$ and $d\neq 3,4,6$. In other words, for $P_\lambda$ given by \eqref{plambda} such that $c_\lambda\in \Q(\zeta_{2d})^\times$ and $c_0,c_2,\ldots,c_n\in \mathbb{S}^1_{\Q(\zeta_{2d})}$ we have
$$
c_{2i-2}^d=-1 \ , \ \ \ \ \text{ for all }i=1,\ldots,\frac{n}{2}+1.
$$
\end{proposition}

\begin{proof}
Let $t\in (\Z/2d\Z)^\times\simeq \Gal(\Q(\zeta_{2d})/\Q)$. Since $\omega_{P_\lambda}$ is a Hodge class, it is a rational class and so it is invariant under the Galois action, i.e. $t(\omega_{P_\lambda})=\omega_{P_\lambda}$. Hence we can write
$$
\omega_{P_\lambda}=c_\lambda\sum_{\beta\in I}\omega_{\beta}\prod_{j=1}^{\frac{n}{2}+1}c_{2j-2}^{\beta_{2j-1}}.
$$
Applying the action of $t$ we get that
$$
t(c_\lambda)\sum_{\beta\in I}t(\omega_\beta)\prod_{j=1}^{\frac{n}{2}+1}t(c_{2j-2}^{\beta_{2j-1}})=c_\lambda\sum_{\beta\in I}\omega_{\beta}\prod_{j=1}^{\frac{n}{2}+1}c_{2j-2}^{\beta_{2j-1}},
$$
and so 
$$
t(c_\lambda)\cdot t(\omega_\beta)\prod_{j=1}^{\frac{n}{2}+1}t(c_{2j-2}^{\beta_{2j-1}})=c_\lambda\cdot \omega_\gamma\prod_{j=1}^{\frac{n}{2}+1}c_{2j-2}^{\gamma_{2j-1}}
$$
for $\omega_\beta\in V(\alpha)$, $\omega_\gamma\in V(t\cdot \alpha)$. It follows from \cref{galact} that
$$
t(c_\lambda)(-1)^{\frac{\sum_{j=1}^{\frac{n}{2}+1}(ta_{2j-2}-\overline{ta_{2j-2}})}{d}}\prod_{j=1}^{\frac{n}{2}+1}t(c_{2j-2}^{d-a_{2j-2}-1})=c_\lambda\prod_{j=1}^{\frac{n}{2}+1}c_{2j-2}^{\overline{-ta_{2j-2}}-1}
$$
holds for all choices of $a_0,a_2,\ldots,a_n\in\{1,\ldots,d-1\}$. For each $j=1,\ldots,\frac{n}{2}+1$, fix the values of $a_{2i-2}=1$ for all $i\neq j$, and let $a_{2j-2}=a,b$ take two variable values $a,b\in\{1,\ldots,d-1\}$. Dividing both resulting identities we obtain 
$$
(-1)^{\frac{ta-tb-\overline{ta}+\overline{tb}}{d}}t(c_{2j-2}^{b-a})=c_{2j-2}^{\overline{-ta}-\overline{-tb}}    
$$
for all $a,b\in\{1,\ldots,d-1\}$. Or equivalently
\begin{equation}
\label{id3}
t(\zeta_{2d}^{a-b}c_{2j-2}^{b-a})=\zeta_{2d}^{\overline{ta}-\overline{tb}}c_{2j-2}^{\overline{tb}-\overline{ta}}.   \end{equation}
Now, let $q:=\min\{p \text{ prime }: \ p\nmid 2d \}$ as in \cref{lemmaent}, hence $\gcd(2d,2d-q)=1$
and $q<\frac{d}{2}$ or $q=\frac{d+1}{2}$. If $q<\frac{d}{2}$, there exists $k\in\{2,3,\dots,d-2\}$ such that $\frac{d}{k+1}<q<\frac{d}{k}$. In this case we have
\begin{equation}
\label{eq:0907}
    (1-k)(\overline{2d-q})-\overline{(k+1)(2d-q)}+k(\overline{2(2d-q)})=-d.
\end{equation}
Using \eqref{id3} for $t=2d-q$ we have
\begin{equation*}
\begin{split}
\zeta_{2d}^{k(2d-q)+\overline{2d-q}-\overline{(k+1)(2d-q)}}c_{2j-2}^{\overline{(k+1)(2d-q)}-\overline{2d-q}}&=t(c_{2j-2}^k)\\
&=\zeta_{2d}^{k\left((2d-q)+\overline{2d-q}-\overline{2(2d-q)}\right)}c_{2j-2}^{k\left(\overline{2(2d-q)}-\overline{2d-q}\right)},
\end{split}
\end{equation*}
and therefore $\zeta_{2d}^{(1-k)(\overline{2d-q})-\overline{(k+1)(2d-q)}+k(\overline{2(2d-q)})}=c_{2j-2}^{(1-k)(\overline{2d-q})-\overline{(k+1)(2d-q)}+k(\overline{2(2d-q)})}.$ By \eqref{eq:0907} we conclude that $c_{2j-2}^d=-1.$ In the case where $q=\frac{d+1}{2}$ the argument above works taking $k=2$ in \eqref{eq:0907}, which is then equal to $d$ instead of $-d$.
\end{proof}




\section{Quadratic fundamental form and proof of \cref{theo1}}
\label{sec6}
In this final section we recall the quadratic fundamental form described by Maclean \cite{maclean2005second}. Her result was described in the context of surfaces for the classical Noether-Lefschetz loci, however in higher dimensions it also gives a partial description of the quadratic fundamental form which is enough for our purposes. Since the original proof applies word by word to the general case we will omit it. 

\begin{definition}
Let $M$ be a smooth $m$-dimensional analytic scheme, $V$ a vector bundle on $M$ and $\sigma$ a section of $V$. Let $W$ be the zero locus of $\sigma$ and let $x\in W$. The \textit{quadratic fundamental form of $\sigma$ at $x$} is
$$
q_{\sigma,x}:T_xW\otimes T_xW\rightarrow V_x/\text{Im}(d\sigma_x)
$$
given in local coordinates $(z_1,\ldots,z_m)$ around $x$ by
$$
q_{\sigma,x}\left(\sum_{i=1}^m\alpha_i\frac{\partial}{\partial z_i},\sum_{j=1}^m\beta_j\frac{\partial}{\partial z_j}\right)=\sum_{i=1}^m\alpha_i\frac{\partial}{\partial z_i}\left(\sum_{j=1}^m\beta_j\frac{\partial}{\partial z_j}(\sigma)\right).
$$
\end{definition}

In our context we will take $M=(T,0)$, $V=\bigoplus_{p=0}^{\frac{n}{2}-1}\mathscr{F}^{p}/\mathscr{F}^{p+1}$ and $x=0$. Where $T\subseteq H^0(\mathcal{O}_{\P^{n+1}}(d))$ is the parameter space of smooth degree $d$ hypersurfaces of $\P^{n+1}$, $\pi: X\rightarrow T$ is the corresponding family, $\mathscr{F}^p=R^n\pi_*\Omega_{X/T}^{\bullet\ge p}$, and $0\in T$ corresponds to the Fermat variety. In order to construct a section $\sigma$ of $V$ around $x$, let $\lambda\in H^{\frac{n}{2},\frac{n}{2}}(X^n_d)_\prim\cap H^n(X^n_d,\Z)$ be a Hodge cycle, and consider $\overline{\lambda}$ its induced flat section in $\mathscr{F}^0/\mathscr{F}^\frac{n}{2}$. If we fix a holomorphic splitting $\mathscr{F}^0/\mathscr{F}^\frac{n}{2}\simeq V$ and we take $\sigma$ as the image of $\overline{\lambda}$ under this splitting, then $W=V_\lambda$. In this context we can identify $T_xW=J^{F,\lambda}_d$ (\cref{propAGtang}), $V_x=\bigoplus_{q=\frac{n}{2}+1}^n R^F_{d(q+1)-n-2}$ and $d\sigma_x=\cdot P_\lambda$. The computation of the degree $d(\frac{n}{2}+2)-n-2$ piece of $q=q_{\sigma,x}$ under these identifications was done by Maclean \cite[Theorem 7]{maclean2005second} as follows.

\begin{theorem}[Maclean]
The degree $r:=d(\frac{n}{2}+2)-n-2$ piece of the fundamental quadratic form
$$
q: \text{Sym}^2(J^{F,\lambda}_d)\rightarrow \bigoplus_{q=\frac{n}{2}+1}^nR^F_{d(q+1)-n-2}/\langle P_\lambda\rangle
$$
is given by
$$
q_r(G,H)=\sum_{i=0}^{n+1}\left(H\frac{\partial Q_i}{\partial x_i}-R_i\frac{\partial G}{\partial x_i}\right)
$$
where 
$$
G\cdot P_\lambda=\sum_{i=0}^{n+1}Q_i\frac{\partial F}{\partial x_i} \ \ \ \text{ and } \ \ \ H\cdot P_\lambda=\sum_{i=0}^{n+1}R_i\frac{\partial F}{\partial x_i}.
$$
\end{theorem}

\begin{proposition}
\label{propqff}
Let $\lambda\in H^{\frac{n}{2},\frac{n}{2}}(X^n_d)_\prim\cap H^n(X^n_d,\Z)$ be a fake linear cycle given by \eqref{plambda}, and consider $$
G:=(x_{2i-2}-c_{2i-2}x_{2i-1})\cdot D\in J^{F,\lambda}_d.
$$
Then
$$
q_r(G,G)=\frac{-c_\lambda}{d}\prod_{j\neq i}\left(\frac{x_{2j-2}^{d-1}-(c_{2j-2}x_{2j-1})^{d-1}}{x_{2j-2}-c_{2j-2}x_{2j-1}}\right)\cdot D^2\cdot (c_{2i-2}^d+1).
$$
\end{proposition}

\begin{proof}
Just note that
$$
G\cdot P_\lambda=c_\lambda\prod_{j\neq i}\left(\frac{x_{2j-2}^{d-1}-(c_{2j-2}x_{2j-1})^{d-1}}{x_{2j-2}-c_{2j-2}x_{2j-1}}\right)\cdot D\cdot (x_{2i-2}^{d-1}-(c_{2i-2}x_{2i-1})^{d-1})
$$
hence $Q_j=0$ for $j\neq 2i-2, 2i-1$ and
$$
Q_{2i-2}=\frac{c_\lambda}{d}\prod_{j\neq i}\left(\frac{x_{2j-2}^{d-1}-(c_{2j-2}x_{2j-1})^{d-1}}{x_{2j-2}-c_{2j-2}x_{2j-1}}\right)\cdot D,
$$
$$
Q_{2i-1}=\frac{-c_\lambda\cdot c_{2i-2}^{d-1}}{d}\prod_{j\neq i}\left(\frac{x_{2j-2}^{d-1}-(c_{2j-2}x_{2j-1})^{d-1}}{x_{2j-2}-c_{2j-2}x_{2j-1}}\right)\cdot D.
$$
The result follows now by a direct computation of Maclean's formula.

\end{proof}

\begin{proof}\textbf{of \cref{theo1}}
After \cref{theo2} we just need to show that 
$$
\text{codim }V_\lambda>{\frac{n}{2}+d\choose d}-\left(\frac{n}{2}+1\right)^2
$$
for all fake linear cycles $\lambda\in H^{\frac{n}{2},\frac{n}{2}}(X^n_d)_\prim\cap H^n(X^n_d,\Z)$. In fact, otherwise $V_\lambda$ is smooth and reduced at Fermat, and so the quadratic fundamental form $q= 0$ vanishes. In particular its degree $r:=d(\frac{n}{2}+2)-n-2$ piece also vanishes $q_r=0$ and so by \cref{propqff} we conclude that $c_{2i-2}^d+1=0$ for all $i=1,\ldots,\frac{n}{2}+1$ contrary to the fact that $\lambda$ is a fake linear cycle.
\end{proof}

\bigskip

\noindent\textbf{Acknowledgements.} The first author was partially supported by Fondecyt ANID postdoctoral grant 3220631. The second author was supported by Fondecyt ANID postdoctoral grant 3210020.


\bibliographystyle{alpha}

\bibliography{ref}



\end{document}